\documentclass{amsart}

\ifx\MyBeamer\undefined
\usepackage[breaklinks,colorlinks]{hyperref}
\usepackage[a4paper,margin=1in]{geometry}
\else
\if\MyBeamer t
\usepackage[breaklinks,colorlinks]{hyperref}
\usepackage[a4paper,margin=1in]{geometry}
\else
\def\emph{\alert}
\fi
\fi

\usepackage{amsmath}
\usepackage{amssymb}
\usepackage{mathrsfs}
\usepackage{amsthm}
\usepackage{aliascnt}
\usepackage[utf8]{inputenc}
\usepackage[T1]{fontenc}
\usepackage[nofancy]{svninfo}
\ifx\FrenchText\undefined
\usepackage[british]{babel}
\else
\usepackage[british,french]{babel}
\AtBeginDocument{%
\catcode`\:=12%
\catcode`\!=12%
}
\fi

\makeatletter

\newcounter{quotethmcnt}
\newenvironment{quotethm}[2][theorem]{\def\quotethmname{quotethm\arabic{quotethmcnt}}\theoremstyle{#1}\newtheorem*{\quotethmname}{\autoref{#2}}\begin{\quotethmname}}{\end{\quotethmname}\stepcounter{quotethmcnt}}

\ifx\MyBeamer\undefined

\def\equationautorefname~#1\null{(#1)}
\def\itemautorefname~#1\null{#1}
\fi

\ifx\MyBeamer\undefined

\ifx\thmnum\undefined
\newtheorem{thmnum}{}[section]
\fi

\newcommand{\mynewthm}[3][]{%
  \newaliascnt{#2}{thmnum}%
  \newtheorem{#2}[#2]{#3}%
  \aliascntresetthe{#2}%
  \newtheorem*{#2*}{#3}%
  \expandafter\newcommand\csname #2autorefname\endcsname{#3}%
  \expandafter\renewcommand\csname the#2\endcsname{\thethmnum}%
}

\newtheorem*{clm}{Claim}
\newenvironment{clmprf}{%
  \begin{proof}[Proof of claim]%
  }{\end{proof}}

\else

\let\xxx=\frametitle
\def\frametitle#1{%
  \xxx{%
    \setbeamercolor*{math text}{use={titlelike,my math text},fg=titlelike.fg!80!my math text.fg}%
    #1}%
  \setbeamercolor{math text}{use=my math text,fg=my math text.fg}%
}

\newcommand{\beamerenv}[3]{%
\newenvironment<>{#1}%
{%
  \setbeamercolor{temp}{fg=structure.fg}%
  \setbeamercolor{structure}{fg=#2}%
  \setbeamercolor{block body}{use=structure,bg=structure.fg!5!white}%
  \begin{#3}%
}%
{\end{#3}\setbeamercolor{structure}{fg=temp.fg}}}

\newcommand{\mynewthm}[3][green!50!black]{%
  \newtheorem*{#2x}{#3}%
  \beamerenv{#2}{#1}{#2x}%
}

\beamerenv{other}{violet}{block}

\fi

\ifx\FrenchText\undefined
\newcommand{\myiffrench}[2]{#2}
\else
\newcommand{\myiffrench}[2]{\iflanguage{french}{#1}{#2}}
\fi

\theoremstyle{plain}
\mynewthm[red]{thm}{\myiffrench{Théorème}{Theorem}}
\mynewthm{prp}{Proposition}
\mynewthm[purple]{lem}{\myiffrench{Lemme}{Lemma}}
\mynewthm[orange]{fct}{\myiffrench{Fait}{Fact}}
\mynewthm[purple!50!white]{cor}{\myiffrench{Corollaire}{Corollary}}

\theoremstyle{definition}
\mynewthm[green!80!black]{dfn}{\myiffrench{Définition}{Definition}}
\mynewthm{hyp}{\myiffrench{Hypothèse}{Hypothesis}}
\mynewthm{conv}{Convention}
\mynewthm{conj}{Conjecture}
\mynewthm{ntn}{Notation}
\mynewthm{cst}{Construction}

\theoremstyle{remark}
\mynewthm[yellow!90!black]{rmk}{\myiffrench{Remarque}{Remark}}
\mynewthm{qst}{Question}
\mynewthm{exc}{\myiffrench{Exercice}{Exercise}}
\mynewthm{exm}{\myiffrench{Exemple}{Example}}

\ifx\MyBeamer\undefined

\newcommand{\myenumlabel}[1]{\textnormal{(\roman{#1})}}

\fi

\newcounter{cycprfcnt}
\newcounter{cycprffirst}
\newcommand{\cycprfpreamble}[1]%
{%
  \setcounter{cycprfcnt}{1}
  \setcounter{cycprffirst}{#1}
  \setlength{\itemindent}{0.5\leftmargin}%
  \setlength{\leftmargin}{0pt}%
  \newcommand{\cpcurr}{\myenumlabel{cycprfcnt}}%
  \newcommand{\cpnext}{\addtocounter{cycprfcnt}{1}\cpcurr}%
  \newcommand{\cpnum}[1]{\setcounter{cycprfcnt}{##1}\cpcurr}%
  \newcommand{\cpfirst}{\cpnum{1}}%
  \newcommand{\impnext}{\cpcurr{} $\Longrightarrow$ \cpnext.}%
  \newcommand{\impfirst}{\cpcurr{} $\Longrightarrow$ \cpfirst.}%
  \def\makelabel##1{\ifnum\value{cycprffirst}=0\hspace{-0.7\itemindent}\setcounter{cycprffirst}{1}\fi##1}%
}%

\newenvironment{cycprf}[1][0]%
{\begin{list}{\impnext}{\cycprfpreamble{#1}}}%
{\qedhere\end{list}}%

\newenvironment{cycprf*}[1][0]%
{\begin{list}{\impnext}{\cycprfpreamble{#1}}}%
{\end{list}}%

\def\indsym#1#2{%
  \setbox0=\hbox{$\m@th#1x$}%
  \kern\wd0%
  \hbox to 0pt{\hss$\m@th#1\mid$\hbox to 0pt{$\m@th#1^{#2}$\hss}\hss}%
  \lower.9\ht0\hbox to 0pt{\hss$\m@th#1\smile$\hss}%
  \kern\wd0}
\newcommand{\ind}[1][]{\mathop{\mathpalette\indsym{#1}}}

\def\nindsym#1#2{%
  \setbox0=\hbox{$\m@th#1x$}%
  \kern\wd0%
  \hbox to 0pt{\hss$\m@th#1\not$\kern1.4\wd0\hss}
  \hbox to 0pt{\hss$\m@th#1\mid$\hbox to 0pt{$\m@th#1^{#2}$\hss}\hss}%
  \lower.9\ht0\hbox to 0pt{\hss$\m@th#1\smile$\hss}%
  \kern\wd0}

\def\dotminussym#1#2{%
  \setbox0=\hbox{$\m@th#1-$}%
  \kern.5\wd0%
  \hbox to 0pt{\hss\hbox{$\m@th#1-$}\hss}%
  \raise.6\ht0\hbox to 0pt{\hss$\m@th#1.$\hss}%
  \kern.5\wd0}

\renewcommand{\emptyset}{\varnothing}
\renewcommand{\setminus}{\smallsetminus}

\usepackage{mathpazo}

\newcommand{\rest}{{\restriction}}

\newcommand{\half}[1][1]{\tfrac{#1}{2}}

\DeclareMathOperator{\tp}{tp}

\newcommand{\Cb}{\mathrm{Cb}}
\DeclareMathOperator{\Th}{Th}

\DeclareMathOperator{\tS}{S}

\DeclareMathOperator{\dcl}{dcl}
\DeclareMathOperator{\acl}{acl}

\DeclareMathOperator{\co}{co}
\newcommand{\cco}{\overline\co}

\newcommand{\id}{\mathrm{id}}

\DeclareMathOperator{\Iso}{Iso}
\DeclareMathOperator{\Aut}{Aut}

\DeclareMathOperator{\Homeo}{Homeo}

\DeclareMathOperator{\supp}{supp}

\DeclareMathOperator{\sgn}{sgn}

\newcommand{\fM}{\mathfrak{M}}
\newcommand{\fN}{\mathfrak{N}}

\newcommand{\fS}{\mathfrak{S}}

\newcommand{\cL}{\mathcal{L}}

\newcommand{\bG}{\mathbf{G}}

\newcommand{\bN}{\mathbf{N}}

\newcommand{\bQ}{\mathbf{Q}}
\newcommand{\bR}{\mathbf{R}}

\newcommand{\bU}{\mathbf{U}}

\newcommand{\ba}{\mathbf{a}}
\newcommand{\bb}{\mathbf{b}}

\newcommand{\bone}{\mathbf{1}}

\makeatother

\usepackage{stmaryrd}

\newcommand{\meq}{\mathrm{meq}}

\begin{document}

\title{On Roeckle-precompact Polish group that cannot act transitively on a complete metric space}

\author{Itaï \textsc{Ben Yaacov}}

\address{Itaï \textsc{Ben Yaacov} \\
  Université Claude Bernard -- Lyon 1 \\
  Institut Camille Jordan, CNRS UMR 5208 \\
  43 boulevard du 11 novembre 1918 \\
  69622 Villeurbanne Cedex \\
  France}

\urladdr{\url{http://math.univ-lyon1.fr/~begnac/}}

\thanks{Author supported by ANR projects GruPoLoCo (ANR-11-JS01-008) and ValCoMo (ANR-13-BS01-0006).}
\thanks{The author wishes to thank Julien \textsc{Melleray} and Todor \textsc{Tsankov} for inspiring discussions, in particular regarding properties of the automorphism group of the Gurarij space.}

\svnInfo $Id: NoTrans.tex 3142 2016-09-19 13:53:21Z begnac $
\thanks{\textit{Revision} {\svnInfoRevision} \textit{of} \today}


\begin{abstract}
  We study when a continuous isometric action of a Polish group on a complete metric space is, or can be, transitive.
  Our main results consist of showing that for certain Polish groups, namely $\Aut^*(\mu)$ and $\Homeo^+[0,1]$, such an action can never be transitive (unless the space acted upon is a singleton).

  We also point out that in all known examples, this pathology coincides with the pathology of Polish groups that are not closed permutation groups and yet have discrete uniform distance, asking whether there is a relation.
  We conclude with a general characterisation/classification of transitive continuous isometric actions of a Roeckle-precompact Polish groups on a complete metric spaces.
  In particular, the morphism from a Roeckle-precompact Polish group to its Bohr compactification is surjective.
\end{abstract}

\maketitle

\tableofcontents

\section*{Introduction}

In this paper we discuss some pathological properties of certain Roelcke-precompact Polish groups.
We shall use as a starting point the following question, asked by Julien \textsc{Melleray}:

\begin{qst*}
  Is there a Polish group that admits no transitive and continuous action by isometries on a non-trivial complete metric space?
  Better still, is there such a group that is Roelcke-precompact (i.e., the automorphism group of an $\aleph_0$-categorical metric structure)?
\end{qst*}

One motivation for the question is a frustrating ``gap'' between theory and practice in the domain of $\aleph_0$-categorical metric structures.
Theory tells us that any $\aleph_0$-categorical structure, Fraïssé limit, separable atomic structure, or separable approximately saturated structure, is \emph{approximately homogeneous}: any two tuples of the same type can be sent \emph{arbitrarily close} to one another by an automorphism (see \cite{BenYaacov-Usvyatsov:dFiniteness,BenYaacov:MetricFraisse}).
On the other hand, in practice, most ``interesting'' structures (all discrete homogeneous structures, as well as the Urysohn space/sphere, the Hilbert space, the atomless probability algebra) are ``precisely homogeneous'', that is to say that tuples of the same type are actually conjugate by an automorphism.

The space of $[0,1]$-valued random variables on an atomless probability space (\cite{BenYaacov:RandomVariables}) is not precisely homogeneous, but is bi-interpretable with the atomless probability algebra, which is precisely homogeneous (that is to say that the two structures are merely two presentations of the same ``underlying mathematical object'', and in particular have the same automorphism group -- see \cite{BenYaacov-Kaichouh:Reconstruction} for more details).
More generally, approximate homogeneity is a robust notion, invariant under bi-interpretation, but precise homogeneity is not, and \emph{any} metric structure (with a non-compact automorphism group) is bi-interpretable with one that is not precisely homogeneous.
Thus the question becomes, is there an $\aleph_0$-categorical structure that is not even bi-interpretable with a precisely homogeneous one?

The set of realisations of a type is a complete set.
Approximate homogeneity means that the automorphism group acts topologically transitively, and even minimally, on that set, while precise homogeneity means that the action is transitive.
Considering structures up to bi-interpretability means, in particular, considering the action on imaginary sorts, namely arbitrary continuous isometric actions on complete metric spaces, whence Melleray's formulation.

We answer this in the affirmative, proving:

\begin{quotethm}{thm:AutStar}
  Let $G = \Aut^*(\mu)$, the group of measure-class-preserving transformations of the Lebesgue measure on $[0,1]$.
  It is naturally isomorphic (as a topological group) to the (isometric) automorphism groups of the Banach lattice $\bigl( L_p(\mu), \wedge, \vee \bigr)$, for any $1 \leq p < \infty$, so it is in particular Roelcke-precompact.

  A continuous action by isometries of $G$ on a complete metric space $X$ cannot be transitive, unless $X$ is a singleton.
\end{quotethm}

This settles Melleray's question and all its variants mentioned above, but opens two new questions.
The first question has to do with the fact that, before \autoref{thm:AutStar} was proved, there were three ``natural candidates'' for a positive example, namely Roelcke-precompact groups for which no transitive action on a complete space was known: $\Aut^*(\mu)$, $\Aut(\bG)$ (the group of linear isometries of the Gurarij space) and $\Homeo^+[0,1]$.
Oddly enough, all three share another pathological property: all three have discrete uniform distance, in the sense of \autoref{sec:DiscreteUniformDistance}, a property that is relatively easy to test for, and are the only ones with this property among the Roelcke pre-compact Polish group the author has encountered so far.
This raises the question of a connection between the two properties, and acts as an incentive for deciding whether $\Aut(\bG)$ and $\Homeo^+[0,1]$ also answer Melleray's question.
The case of $\Homeo^+[0,1]$ is studied in \autoref{sec:H01}, again with an affirmative answer.

The second question is more of an aesthetic nature: while the main assertion of \autoref{thm:AutStar} is purely topological, the argument makes heavy use of model theory, in at least two points.
The first is \autoref{prp:Aleph0CategoricalNamedACL}, i.e., the passage from $\aleph_0$-categoricity over a named parameter to $\aleph_0$-categoricity over its algebraic closure.
This is a general fact, which can be be translated to topological terms as:
\begin{quotethm}{thm:SurjectiveBohrCompactification}
  Let $G$ be a Roelcke-precompact Polish group and $\varphi\colon G \rightarrow B$ its Bohr compactification.
  Then $\varphi$ is surjective.
\end{quotethm}
In \autoref{sec:TransitiveAction} we give a general characterisation of continuous isometric transitive actions of a Roelcke-precompact group on a complete metric space, from which this follows (subsequent to the circulation of a draft of the present paper, Todor \textsc{Tsankov} announced a more direct proof).

Model theory comes in a second time, in the use of stability in order to get weak elimination of metric imaginaries (WEMI) in atomless $L_p$ Banach lattices.
We content ourselves with a translation of WEMI to topological terms (\autoref{lem:Aleph0CategoricalEMI}), advising the reader that model-theoretic stability has a clear topological counterpart in the form of weak almost periodicity -- see \cite{BenYaacov-Tsankov:WAP,BenYaacov:Grothendieck}.

\section{Metric imaginaries}
\label{sec:Imaginaries}

We start with a few general reminders regarding imaginary sorts in metric structures.
For this we assume some familiarity with the model theory of metric structures as exposed in \cite{BenYaacov-Usvyatsov:CFO} or \cite{BenYaacov-Berenstein-Henson-Usvyatsov:NewtonMS}.
In the specific case of $\aleph_0$-categorical structures one can give an alternative definition (\autoref{fct:Aleph0CategoricalImaginary}), and the reader who is willing to accept a few black boxes may skip there directly.

\begin{dfn}
  \label{dfn:MetricImaginaries}
  Let $\fM$ be a metric structure, say in a single-sorted language $\cL$.
  Let $\rho$ be a definable (without parameters) pseudo-distance on $\fM^\bN$ (we shall consider any pseudo-distance on $\fM^n$ as a pseudo-distance on $\fM^\bN$ through the addition of dummy variables).
  Define an equivalence relation $\xi \sim_\rho \zeta \Longleftrightarrow \rho(\xi,\zeta) = 0$ (the \emph{kernel} of $\rho$), and let $(\fM_\rho,d)$ be the completion of the metric space $(\fM^\bN/{\sim_\rho},\rho)$.
  We call $\fM_\rho$ a \emph{metric imaginary sort} and a member thereof a \emph{metric imaginary}.
  We define $\fM^\meq$ to be the disjoint union of all metric imaginary sorts of $\fM$ (identifying $\fM$ with $\fM_d$).

  We observe that any elementary embedding $\fM \hookrightarrow \fN$ induces an isometric embedding map $\fM_\rho \hookrightarrow \fN_\rho$.
  This is true in particular for automorphisms of $\fM$, giving rise to a continuous isomteric action $\Aut(\fM) \curvearrowright \fM^\meq$.

  For each $\rho$ and $n$ we define a predicate symbol $P_{\rho,n}(x,y)$ on $\fM^n \times \fM_\rho$ by:
  \begin{gather*}
    P_{\rho,n}(x,y) = \inf_{x' \in \fM^\bN} \, \rho(xx',y).
  \end{gather*}
  We let $\cL^\meq$ denote the original language $\cL$ together with the new predicate symbols (bounds and continuity moduli for $P_{\rho,n}$ can be deduced from those of $\rho$).

  The same definitions can be extended to a multi-sorted structure, in which case a definable pseudo-distance, or any definable predicate for that matter, is defined on some countable (possibly finite) product of sorts.
\end{dfn}

\begin{rmk}
  We may restrict the definition to $\rho$ that define a pseudo-distance in \emph{every} $\cL$-structure.
  Indeed, if $\varphi(x,y)$ is any definable predicate then $\rho_\varphi(x,y) = \sup_z\, \bigl| \varphi(x,z) - \varphi(y,z) \bigr|$ always defines a pseudo-distance; if $\varphi$ already defines a pseudo-distance in $\fM$ then it agrees there with $\rho_\varphi$.
  Thus $\cL^\meq$ need only depend on $\cL$ and not on $\fM$.

  In \cite{BenYaacov-Usvyatsov:CFO} we defined imaginary only for such $\rho_\varphi$ where $z$ is a finite tuple, so the definition here is slightly more general.
\end{rmk}

One easily checks that:
\begin{enumerate}
\item The structure $\fM^\meq$ is interpretable in $\fM$ in the sense of \cite{BenYaacov-Kaichouh:Reconstruction}.
\item The $\cL^\meq$-elementary extensions of $\fM^\meq$ are exactly those of the form $\fN^\meq$ where $\fN \succeq \fM$.
\item If the theory of $\fM$ is model-complete then so is the $\cL^\meq$-theory of $\fM^\meq$.
\end{enumerate}

\begin{dfn}
  \label{dfn:MetricDACL}
  Let $A \subseteq \fM^\meq$ be some set.
  We say that $b \in \fM^\meq$ is \emph{definable} (respectively, \emph{algebraic}) over $A$ if for every elementary extension $\fN \succeq \fM$ the set $\{c \in \fN^\meq : c \equiv_A b\}$ consists of $b$ alone (is compact).

  The \emph{definable closure} $\dcl^\meq(A) \subseteq \fM^\meq$ (respectively, the \emph{algebraic closure} $\acl^\meq(A) \subseteq \fM^\meq$) is the collection of all imaginaries $b$ definable (respectively, algebraic) over $A$.
\end{dfn}

If $A \subseteq \fM^\meq$ and $\fN \succeq \fM$, then $\dcl^\meq(A)$ evaluates in $\fM$ and in $\fN$ to the same set (a subset of $\fM^\meq$).
Also, $b$ is definable (algebraic) over $A$ only if it is definable (algebraic) over some countable subset $A_0 \subseteq A$.

\begin{lem}
  \label{lem:Countable}
  \begin{enumerate}
  \item Every countable $A \subseteq \fM^\meq$ is interdefinable with some singleton $b \in \fM^\meq$ (i.e., $\dcl^\meq(A) = \dcl^\meq(b)$).
  \item If $\fM$ is separable in a countable language then every set $A \subseteq \fM^\meq$ is interdefinable with some countable subset $A_0 \subseteq A$, and thus with some singleton $b \in \fM^\meq$
  \end{enumerate}
\end{lem}
\begin{proof}
  The first item is fairly standard.
  For the second, assume not.
  Then we can find a sequence $(a_i)_{i < \aleph_1} \subseteq A$ such that each $a_i$ is not definable over the ones preceding it.
  Fix a countable dense sequence $\bb = (b_n)_{n \in \bN} \subseteq \fM$.
  For $i < \aleph_1$ let $\ba_i = (a_j)_{j<i}$ and let $\pi_i(x,y)$ be the partial type in $2 \times \bN$ variables saying that there exists a sequence $z$ of length $i$ such that $xz \equiv yz \equiv \bb\ba_i$, noticing that if such $z$ exists it must be unique.
  Fix $i$.
  By hypothesis there exists (in a sufficiently saturated elementary extension) $a' \equiv_{\ba_i} a_i$ such that $a' \neq a_i$.
  Let $\ba' = (\ba_i,a') \equiv \ba_{i+1}$, and find $\bb'$ such that $\bb'\ba' \equiv \bb\ba_{i+1}$.
  Then $\pi_i(\bb,\bb')$ holds but $\pi_{i+1}(\bb,\bb')$ fails, so each $\pi_i$ is strictly stronger the the preceding ones.
  This defines a strictly decreasing sequence of closed sets in $\tS_{2 \times \bN}(\emptyset)$, which is therefore not second-countable.
  This is impossible in a countable language.
\end{proof}

\begin{dfn}
  \label{dfn:EliminationMetricImaginaries}
  \begin{enumerate}
  \item We say that $\fM$ has \emph{elimination of metric imaginaries (EMI)} if for every $a \in \fM^\meq$ there exists a subset $A \subseteq \fM$ that is interdefinable with $a$, namely, such that $A \subseteq \dcl^\meq(a)$ and $a \in \dcl^\meq(A)$, i.e., such that $\dcl^\meq(a) = \dcl^\meq(A)$.
  \item We say that $\fM$ has \emph{weak elimination of metric imaginaries (WEMI)} if for every $a \in \fM^\meq$ there exists a subset $A \subseteq \fM$ such that $A \subseteq \acl^\meq(a)$ and $a \in \dcl^\meq(A)$.
    Equivalently, if for every $a \in \fM^\meq$ there exists a subset $A \subseteq \fM$ such that $\acl^\meq(a) = \dcl^\meq(A)$.
  \item We say that a theory $T$ has EMI or WEMI if all its models do.
  \end{enumerate}
\end{dfn}

When specialised to classical logic, (W)EMI is equivalent to elimination of hyperimaginaries (in the sense of \cite{Buechler-Pillay-Wagner:SupersimpleTheories}, i.e., down to imaginaries) plus ordinary (weak) elimination of imaginaries.
For the weak case some argument is required, see \cite{Buechler-Pillay-Wagner:SupersimpleTheories}.

\begin{fct}
  \label{fct:StableWEMI}
  Let $T$ be a stable theory such that every type of a tuple in the home sort over a model admits a canonical base in the home sort (or in some family of distinguished sorts).
  Then $T$ admits WEMI (down to the distinguished family of sorts).
\end{fct}
\begin{proof}
  This is essentially folklore but we give a quick proof for completeness.

  Let $\fM \vDash T$ and $a \in \fM^\meq$.
  Then $\fM^\meq$, being interpretable in $\fM$, is stable, and we can find $\fN \succeq \fM$ and $\fM_1 \preceq \fN$ such that $\fM \equiv_a \fM_1$ (so $a \in \fM_1^\meq$) and $\fM_1 \ind_a \fM$.
  Let $A = \Cb(\fM_1/\fM)$ as a set in the home sort (or in the distinguished family of sorts).
  Then first, $A \subseteq \fM$.
  Second, $\fM_1 \ind_a \fM$ implies that $A \subseteq \acl^\meq(a)$.
  Finally, since $a \in \fM_1^\meq$, the type $a/\fM$ is definable over $A$.
  Since $a \in \fM^\meq$ as well, $\Cb(a/\fM) = a$, so $a \in \dcl^\meq(A)$.
\end{proof}

\begin{dfn}
  \label{dfn:MetricGSpace}
  Let $G$ be a topological group.
  By a \emph{metric $G$-space} we shall mean a non-empty metric space $X$ together with an action $G \curvearrowright X$ that is jointly continuous and acts by isometry.
  If $X$ is complete, we call it a \emph{complete $G$-space}.
  Following \cite{BenYaacov-Tsankov:WAP}, for $x \in X$ we let $[x] = \overline{G x}$ denote the orbit closure of $x$.
  The set of orbit closures $X \sslash G = \bigl\{ [x] : x \in X \bigr\}$ forms a partition of $X$.
  We equip $X \sslash G$ with the set-distance function induced from $X$.
  It is always a metric space, and is complete if $X$ is.

  Morphisms between $G$-spaces are continuous maps that respect the action.

  We say that a complete $G$-space $X$ is \emph{approximately oligomorphic} if $X^n \sslash G$ is compact for all $n$.
\end{dfn}

When $G$ is metrisable (e.g., Polish) then it admits a compatible left-invariant distance $d_L$, unique up to uniform equivalence.
The \emph{left completion} of $G$ is the completion $\widehat{G}_L = \widehat{(G,d_L)}$.
It is naturally a complete $G$-space under the left action of $G$, and the action of $G$ on any complete $G$-space extends continuously to an action of $\widehat{G}_L$ (whose action on itself makes it a topological semi-group).
Notice that a metric $G$-space $X$ is minimal if and only if $X \sslash G$ is a singleton if and only if $G \curvearrowright X$ is topologically transitive (equivalently, minimal).
In particular, $G \curvearrowright \widehat{G}_L$ is minimal.

\begin{dfn}
  A topological group $G$ is \emph{Roelcke-precompact} if for every nonempty open $U \subseteq G$ there exists a finite set $F \supseteq G$ such that $UFU = G$.
  Equivalently, if its \emph{Roelcke completion} $R(G) = \widehat{G}_L^2 \sslash G$ is compact.
\end{dfn}

By \cite{BenYaacov-Tsankov:WAP}, a Polish group $G$ is Roelcke-precompact if and only if it is the automorphism group of an $\aleph_0$-categorical metric structure if and only if, for a complete $G$-space $X$ to be approximately oligomorphic it suffices that $X \sslash G$ be compact.
More generally, if both $X \sslash G$ and $Y \sslash G$ are compact (for a Roelcke-precompact $G$) then so is $(X \times Y) \sslash G$.

\begin{fct}[Ryll-Nardzewski/Henson Theorem for metric structures, see \cite{BenYaacov-Usvyatsov:dFiniteness} for a more detailed statement]
  \label{fct:RyllNardzewski}
  A structure $\fM$ is $\aleph_0$-categorical if and only if $G \curvearrowright \fM$ is approximately oligomorphic, where $G = \Aut(\fM)$.
  In this case every continuous $G$-invariant predicate is definable.
\end{fct}

\begin{lem}
  \label{lem:RoelckePrecompactAmalgamation}
  Let $G$ be a Roelcke-precompact Polish group.
  Let $X_i$ be minimal complete $G$-spaces and $x_i,y_i \in X_i$ for $i \in \bN$.
  Then there exist $\xi_i \in \widehat{G}_L$ such that $\xi_i x_i = \xi_{i+1} y_i$ for all $i$.

  A commonly used special case is the following: let $I$ be countable, $X$ a minimal complete $G$-space, and $x_i \in X$ for $i \in I$.
  Then there exist $\xi_i \in \widehat{G}_L$ and $z \in X$ such that $z = \xi_i x_i$ for all $i \in I$.
\end{lem}
\begin{proof}
  For each $i$ there exist $g_{i,n} \in G$ such that $g_{i,n} y_i \rightarrow x_i$.
  Choose $f_{i,n} \in G$ such that $f_{i+1,n} = f_{i,n} g_{i,n}$, so $d(f_{i+1,n} y_i,f_{i,n} x_i) \rightarrow 0$ as $n \rightarrow \infty$.
  By compactness of $\widehat{G}_L^\bN \sslash G$ we may assume that $[f_{0,n},f_{1,n},\ldots] \rightarrow [\xi_0,\xi_1,\ldots]$, and the $\xi_i$ are as desired.

  The special case follows since we may assume that $I = \bN$ and take $y_i = x_{i+1}$.
\end{proof}

\begin{prp}
  \label{prp:GSpaceMorphism}
  Let $G$ be a Roelcke-precompact Polish group, $X$ and $Y$ minimal complete $G$-spaces, and let $\pi\colon Y \rightarrow X$ be a morphism of $G$-spaces.
  \begin{enumerate}
  \item The map $\pi$ is uniformly continuous.
  \item If $\pi$ is injective then it is surjective and $\pi^{-1}$ is uniformly continuous as well.
  \item If all fibres of $\pi$ are compact (possibly empty) then the fibre map $x \mapsto \pi^{-1} x$ is uniformly continuous in the Hausdorff distance.
    In particular, $\pi$ is surjective and all its fibres are isometric.
  \end{enumerate}
\end{prp}
\begin{proof}
  Since $Y \sslash G$ is a singleton and $G$ Roelcke-precompact, $Y^2 \sslash G$ is compact as well.

  Assume $\pi$ is not uniformly continuous, i.e., that there exist $(y_n,z_n) \in Y^2$ such that $d(y_n,z_n) \rightarrow 0$ and yet $d(\pi y_n, \pi z_n) \geq \varepsilon > 0$ for all $n$.
  Possibly passing to a sub-sequence, we may assume that $[y_n,z_n] \rightarrow [y,z] \in Y^2 \sslash G$, i.e., that there exist $g_n \in G$ such that $g_n y_n \rightarrow y$ and $g_n z_n \rightarrow z$.
  Since the actions are isometric and commute with $\pi$ we may replace $(y_n,z_n)$ with $(g_n y_n, g_n z_n)$, so $(y_n,z_n) \rightarrow (y,z)$.
  On the one hand, $d(y_n,z_n) \rightarrow 0$ implies $y = z$, and on the other hand $d(\pi y, \pi z) \geq \varepsilon$, a contradiction.

  Assume now that $\pi$ is injective, but there exist $(y_n,z_n) \in Y^2$ such that $d(y_n,z_n) \geq \varepsilon > 0$ and yet $d(\pi y_n, \pi z_n) \rightarrow 0$.
  As above we may assume that $(y_n,z_n) \rightarrow (y,z)$, so $d(y,z) \geq \varepsilon$ and $\pi y = \pi z$, contradicting our hypothesis.
  We conclude that $\pi^{-1}$ is uniformly continuous where defined.
  Now let $x \in X$ and $y \in Y$ be arbitrary.
  By minimality, there exists a sequence $g_n \in G$ such that $g_n \pi y = \pi g_n y \rightarrow x$.
  The sequence $(g_n y)$ is therefore Cauchy and converges to some $z \in Y$.
  By continuity, $\pi z = x$, so $\pi$ is surjective.

  Finally, assume that all fibres of $\pi$ are compact.
  Let $Z_0 = K(Y)$ be the space of non-empty compact sub-spaces of $Y$ equipped with the Hausdorff distance and the natural action of $G$.
  It is straightforward to check that $Z_0$ is a complete $G$-space.
  Choose any $x \in X$ such that $\pi^{-1} x \neq \emptyset$ (one such must exist since $Y \neq \emptyset$) and let $Z_x = [\pi^{-1} x] \subseteq Z_0$.
  Then $Z_x$ is a minimal complete $G$-space, and the image under $\pi$ of any $S \in Z_x$ is a singleton in $X$, giving rise to a map $\overline{\pi} \colon Z_x \rightarrow X$.
  One easily checks that it is injective, continuous and respects the action.
  By the second item, $\overline{\pi}$ is bijective and uniformly continuous in both direction, and everything else follows.
\end{proof}

\begin{fct}
  \label{fct:Aleph0CategoricalImaginary}
  Let $\fM$ be $\aleph_0$-categorical and $G = \Aut(\fM)$.
  Let also $a \in \fM^\meq$ and $X = \{b \in \fM^\meq : b \equiv a\}$.
  Then $X$ is a minimal complete $G$-space (so in particular $X = [a]$).
  Conversely, every minimal complete $G$-space arises (up to an isometric bijection) in this fashion.

  Moreover, if $a,a' \in \fM^\meq$, then
  \begin{enumerate}
  \item We have $a' \in \dcl^\meq(a)$ if and only if there exists a (necessarily unique) morphism $\pi\colon [a] \rightarrow [a']$ sending $a \mapsto a'$.
  \item We have both $a' \in \dcl^\meq(a)$ and $a \in \acl^\meq(a')$ if and only if, in addition, $\pi$ has compact fibres.
  \end{enumerate}
\end{fct}
\begin{proof}
  The main assertion is easy.
  For the converse see \cite[Proposition~4.2]{BenYaacov-Melleray:Grey}.
  The moreover part is standard model theory.
\end{proof}

Thus an imaginary element in an $\aleph_0$-categorical structure $\fM$ is the same thing as a distinguished point in a minimal $\Aut(\fM)$-space.

\begin{lem}
  \label{lem:Aleph0CategoricalEMI}
  Let $\fM$ be $\aleph_0$-categorical and $G = \Aut(\fM)$.
  Then $\fM$ has EMI (WEMI) if and only if for every minimal complete $G$-space $X$ there exists a countable tuple $a \in \fM^\bN$ and a morphism of $G$-spaces $\pi\colon [a] \rightarrow X$ such that $\pi$ is injective (has compact fibres).
\end{lem}
\begin{proof}
  By \autoref{prp:GSpaceMorphism} and \autoref{fct:Aleph0CategoricalImaginary}.
\end{proof}

For a structure $\fM$ and a subset $A \subseteq \fM^\meq$, let $(\fM,A)$ denote $\fM$ with all members of $A$ named (one does not have to add new sorts for this -- rather, for each $a \in A \cap \fM_\rho$ name all the predicates $P_{\rho,n}(x,a)$, and by \autoref{lem:Countable} it suffices to name countably many such predicates).

\begin{lem}
  \label{lem:Aleph0CategoricalNamedTransitive}
  Let $\fM$ be an $\aleph_0$-categorical structure in a countable language and $a \in \fM^\meq$.
  The structure $(\fM,a)$ is $\aleph_0$-categorical if and only if $G \curvearrowright [a]$ is transitive.
\end{lem}
\begin{proof}
  Since $\fM$ is $\aleph_0$-categorical, every separable model of $\Th(\fM,a)$ is isomorphic to $(\fM,b)$ for some $b \in [a]$, and conversely, every such $(\fM,b)$ is a model.
  Thus $(\fM,a)$ is $\aleph_0$-categorical if and only if $(\fM,a) \cong (\fM,b)$ for every $b \in [a]$, i.e., if and only if $G \curvearrowright [a]$ is transitive.
\end{proof}

\begin{prp}
  \label{prp:Aleph0CategoricalNamedACL}
  Let $\fM$ be an $\aleph_0$-categorical structure in a countable language.
  Then $(\fM,A)$ is $\aleph_0$-categorical for any $A \subseteq \acl^\meq(\emptyset)$.

  Consequently, if $(\fM,A)$ is $\aleph_0$-categorical for some $A \subseteq \fM^\meq$ then $(\fM,B)$ is $\aleph_0$-categorical for any $B \subseteq \acl^\meq(A)$.
\end{prp}
\begin{proof}
  Define a binary predicate $P(x,y)$ as the distance between $\tp(x/A)$ and $\tp(y/A)$.
  Then $P$ is metrically continuous and automorphism-invariant on $\fM$, so by \autoref{fct:RyllNardzewski} it is definable without parameters.
  Let $G_A = \Aut(\fM,A) = \{g \in G : g\rest_A = \id_A\}$.
  If the space $\fM^k \sslash G_A$ of $k$-types over $A$ is not totally bounded, there exist $\varepsilon > 0$ and a sequence $(a_n)_{n \in \bN}$ in $\fM^k$ such that $P(a_n,a_m) \geq \varepsilon$ for all $n < m$.
  By standard arguments we may assume that the sequence $(a_n)$ is indiscernible, in which case it is also indiscernible over $\acl^\meq(\emptyset)$, contradicting $P(a_0,a_1) > 0$.

  Thus $G_A \curvearrowright \fM$ is approximately oligomorphic and $(\fM,A)$ is $\aleph_0$-categorical.
\end{proof}

\section{Triviality of transitive actions}
\label{sec:TrivialAction}

\begin{thm}
  \label{thm:NoTrans}
  Let $\fM$ be an $\aleph_0$-categorical structure, $G = \Aut(\fM)$.
  Assume that for every $a \in \fM$, either $[a]$ is a singleton (i.e., $a \in \dcl(\emptyset)$) or $G \curvearrowright [a]$ is not transitive.
  Assume moreover that $\fM$ has WEMI.

  Then a transitive complete $G$-space is necessarily a singleton.
\end{thm}
\begin{proof}
  Assume $G \curvearrowright X$ is a transitive complete $G$-space.
  By \autoref{fct:Aleph0CategoricalImaginary} we may assume that $X = [a]$ for some $a \in \fM^\meq$.
  By \autoref{lem:Aleph0CategoricalNamedTransitive}, $(\fM,a)$ is $\aleph_0$-categorical.
  By \autoref{prp:Aleph0CategoricalNamedACL}, $\bigl( \fM, \acl^\meq(a) \bigr)$ is $\aleph_0$-categorical.
  By WEMI there exists a set $B \subseteq \fM$ interdefinable with $\acl^\meq(a)$, so $(\fM,B)$ is $\aleph_0$-categorical.
  In particular, $(\fM,b)$ is $\aleph_0$-categorical for any $b \in B$.
  By hypothesis (and \autoref{lem:Aleph0CategoricalNamedTransitive}) we must have $b \in \dcl(\emptyset)$.
  Thus $B \subseteq \dcl(\emptyset)$, and therefore $a \in \dcl^\meq(\emptyset)$.
  We conclude that $X = \{a\}$.
\end{proof}

For every (non-compact) Roelcke-precompact Polish group $G$ there exists a structure $\fM$ satisfying all the hypotheses of \autoref{thm:NoTrans} with the possible exception of WEMI\@.
Indeed, let $\widehat{G}_L = \widehat{(G,d_L)}$ be the left completion of $G$.
For each $n$ and a countable dense family of $a \in G^n$, define a predicate $P_a(x) = d\bigl(x, [a]\bigr)$.
Then $\widehat{G}_L$ together with these predicates forms a separable metric structure $\fM_G$ such that $G = \Aut(\fM)$; if $G$ is Roelcke-precompact then $\fM_G$ is $\aleph_0$-categorical (see \cite{Melleray:Oscillation,BenYaacov-Tsankov:WAP}).
If $G$ is not compact then $G \neq \widehat{G}_L$, so $G$ does \emph{not} act transitively on $\widehat{G}_L$.
On the other hand, since $G$ is dense in $\widehat{G}_L$, we have $\widehat{G}_L \sslash G = \{*\}$, i.e., $[a] = \widehat{G}_L$ for every $a \in \widehat{G}_L$.



In general, if $\fM$ is $\aleph_0$-categorical and $G = \Aut(\fM)$, then $\fM$ is interpretable in $\fM_G$ (see \cite{BenYaacov-Kaichouh:Reconstruction}), and under this interpretation any $a \in \widehat{G}_L$ codes an entire copy of $\fM$.
For example, if $G = \Aut(\ell_2)$ is the unitary group, then members of $\widehat{G}_L$ code infinite-dimensional spaces, so any single vector, viewed as an imaginary in $\fM_G$, witnesses that $\fM_G$ fails WEMI -- which is to be expected, since $G$ does act transitively on the unit ball of $\ell_2$.

Thus, in a sense, the ``tricky'' property is WEMI.
One method for proving WEMI is using \autoref{fct:StableWEMI}, which leads us to the following example.

\begin{fct}
  \label{fct:ALpL}
  Let $1 \leq p < \infty$ and let $ALpL$ be the theory of atomless $L_p$ Banach lattices (or more precisely, of their closed unit balls) in the language $\{0,-,\half[x+y], |{\cdot}|\}$.
  Then
  \begin{enumerate}
  \item The theory $ALpL$ is complete and $\aleph_0$-categorical: all its separable models are isomorphic to $\bigl( L_p(\mu),|{\cdot}| \bigr)$, where $\mu$ denotes the Lebesgue measure on $[0,1]$ and $L_p(\mu)$ is the space of real-valued $p$-summable functions.
    On the other hand, for any single non-zero member $f$ of the home sort, the structure $(\fM,f)$ is not $\aleph_0$-categorical.
  \item The theory $ALpL$ is $\aleph_0$-stable, and admits canonical bases in the home sort.
    Consequently, it has WEMI.
  \end{enumerate}
\end{fct}
\begin{proof}
  For categoricity, stability, canonical bases and so on see \cite{BenYaacov-Berenstein-Henson:LpBanachLattices,BenYaacov:UniformCanonicalBases} (the former uses a formalism that is different from, although equivalent to, our formalism of continuous logic).
  If $f \neq 0$ then $\Th(\fM,f)$ admits at least (in fact, exactly) two non-isomorphic separable models, one in which the support of $f$ has full measure and one in which it does not (i.e., one in which $|f| \wedge |g| = 0$ implies $g = 0$, and one in which it does not).
\end{proof}

\begin{dfn}
  \label{dfn:AutStar}
  Let $\mu$ denote the Lebesgue measure on $[0,1]$.
  We define the group $\Aut^*(\mu)$ to consist of measure-class-preserving transformations of $\mu$, i.e., of bijections $\varphi$ between full measure subsets of $[0,1]$ such that both $\varphi$ and $\varphi^{-1}$ are measurable and respect the null-measure ideal, identified up to equality a.e.\ (see also Pestov \cite[pp.~91--92]{Pestov:Dynamics}).

  Fix $1 \leq p < \infty$.
  For $\varphi \in \Aut^*(\mu)$ and $f \in L_p(\mu)$ define
  \begin{gather*}
    \varphi \cdot_p f = \left( \frac{d \varphi_* \mu}{d \mu}\right)^{\frac{1}{p}} \varphi_* f
  \end{gather*}
  where $\varphi_* \mu(A) = \mu\bigl( \varphi^{-1}(A) \bigr)$ is the image measure, $d \varphi_* \mu / d \mu$ is the Radon-Nikodym derivative, and $\varphi_* f = f \circ \varphi^{-1}$.
  We observe that $\varphi \cdot_p f \in L_p(\mu)$ as well and the the action is isometric.
  We equip $\Aut^*(\mu)$ with the topology of point-wise convergence of its action on, say, $L_2(\mu)$.
\end{dfn}

\begin{thm}
  \label{thm:AutStar}
  \begin{enumerate}
  \item For any $1 \leq p < \infty$, the action $\Aut^*(\mu) \curvearrowright L_p(\mu)$ is by linear isometries, inducing a homeomorphic isomorphism $\Aut^*(\mu) \cong \Aut\bigl( L_p(\mu), |{\cdot}| \bigr)$.
    In particular, $\Aut^*(\mu)$ is a Polish and Roelcke-precompact topological group.
  \item Any transitive complete $\Aut^*(\mu)$-space is a singleton.
  \end{enumerate}
\end{thm}
\begin{proof}
  The action clearly respects the Banach lattice structure, and is isometric.
  This yields a group monomorphism $\Aut^*(\mu) \hookrightarrow \Aut\bigl( L_p(\mu), |{\cdot}| \bigr)$.
  Assume conversely that $\varphi \in \Aut\bigl( L_p(\mu), |{\cdot}| \bigr)$.
  For measurable $A \subseteq [0,1]$ let $\tilde{\varphi}(A) = \supp \varphi(\bone_A)$ (or, equivalently, the support of $\varphi(f)$ where $\supp f = A$).
  Define
  \begin{gather*}
    \psi\colon [0,1] \rightarrow [0,1], \qquad \psi(r) = \inf \, \bigl\{ q \in [0,1] \cap \bQ : r \in \tilde{\varphi}\bigl([0,q]\bigr) \bigr\},
  \end{gather*}
  so $\psi^{-1}(A) = \tilde{\varphi}(A)$ a.s.\ for all measurable $A$.
  Then $\psi \in \Aut^*(\mu)$ and $\psi \cdot_p f = \varphi(f)$, first for $f$ of the form $\varphi^{-1}(\bone_A)$ and then for arbitrary $f$, whence the isomorphism $\Aut^*(\mu) \cong \Aut\bigl( L_p(\mu), |{\cdot}| \bigr)$.
  It is by definition homeomorphic for $p = 2$, so $\Aut^*(\mu)$ is the topological automorphism group of a separable $\aleph_0$-categorical structure.
  It is therefore Polish and Roelcke-precompact (see \cite{BenYaacov-Tsankov:WAP}).

  Since $\bigl( L_p(\mu), |{\cdot}| \bigr)$ is definable in $\bigl( L_2(\mu), |{\cdot}| \bigr)$ via the map $f \rightarrow \sgn(f) \cdot |f|^{2/p}$ (see \cite[Lemma~3.3]{BenYaacov:UniformCanonicalBases}), and since this sends the action $\cdot_2$ to $\cdot_p$, the topologies agree for all $p$.

  The second item is by \autoref{thm:NoTrans} and \autoref{fct:ALpL}.
\end{proof}

\begin{rmk}
  \label{rmk:AutStar}
  Since $\bigl( L_p(\mu), |{\cdot}|\bigr)$ is stable, the WAP compactification of $\Aut^*(\mu)$ agrees with its Roelcke compactification.
\end{rmk}

Another natural potential example is the Gurarij space $\bG$ (see \cite{Gurarij:UniversalPlacement,Lusky:UniqueGurarij} for the general theory and \cite{BenYaacov-Henson:Gurarij} for the model-theoretic aspects).
It is $\aleph_0$-categorical and eliminates quantifiers in the natural Banach space language, so the orbit closure of any $a \in \bG \setminus \{0\}$ is the sphere of radius $\|a\|$.
Since it contains both smooth and non-smooth vectors, the isometric automorphism group $\Aut(\bG)$ does \emph{not} act transitively on the sphere.
WEMI is significantly more complicated, though: the Gurarij space is as far from stable as possible (it has TP2), so \autoref{fct:StableWEMI} is of no help.
On the other hand, it seems structurally simple enough that no ``complicated'' imaginary sorts should exist, and we expect that an explicit analysis of imaginary sorts is possible.

\begin{conj}
  The Gurarij space $\bG$ has WEMI and every transitive complete $\Aut(\bG)$-space is a singleton.
\end{conj}

\section{The case of $G = \Homeo^+[0,1]$}
\label{sec:H01}

In this section we concentrate on the group $G = \Homeo^+[0,1]$ of increasing self-homeomorphisms of the interval, equipped with the topology of uniform convergence.
It is Roelcke-precompact (see Megrelishvili \cite{Megrelishvili:EverySemitopologicalSemigroup}, this will also follow easily from the present analysis), and like $\Aut(\mu)$ and $\Aut^*(\mu)$, is given, at least at a first time, via its action on the ``dual space''.
While its topology coincides with point-wise convergence of its action on $C[0,1]$, this action is not approximately oligomorphic -- that is to say that $C[0,1]$ is ``too big''.
One possible solution is to restrict the action to a single (or few) orbit closures, e.g., to $[\id_{[0,1]}]$ -- and this just boils down to the tautological action $G \curvearrowright \widehat{G}_L$.
In other words, the most natural $\aleph_0$-categorical $\fM$ for which $G = \Aut(\fM)$ is, as far as we can see, none other than the ``tautological'' structure $\fM_G$ referred to earlier.
Exceptionally for this group, the tautological structure turns out to be useful.

Let $\fM$ be the space of all continuous, weakly increasing surjective maps $\xi\colon [0,1] \rightarrow [0,1]$.
Equipped with the metric of uniform convergence $d(\xi,\zeta) = \|\xi-\zeta\|_\infty$ it is complete, and it is naturally a metric $G$-space with the action $g \cdot \xi = \xi \circ g^{-1}$.
Notice that the embedding $G \hookrightarrow \fM$, $g \rightarrow g^{-1}$ induces a left-invariant compatible distance on $G$, and its image in $\fM$ is dense, so $\fM$ may indeed be viewed as the left completion $\widehat{G}_L$, or, if we choose the more natural identity embedding, the right completion $\widehat{G}_R$.
As per the remarks following \autoref{thm:NoTrans}, we may name countable many predicates on $\fM$ so that $G = \Aut(\fM)$.

Let us describe the set $\fM^n \sslash G$ (the space of $n$-types) explicitly.

\begin{dfn}
  \label{dfn:H01TypeSpaces}
  \begin{enumerate}
  \item Given $\xi \in \fM$ we may define a partial inverse $\xi^*\colon [0,1] \rightarrow [0,1]$ such that $\xi \circ \xi^* = \id$ and $\xi^*$ is left-continuous (which entirely determines $\xi^*$).
  \item For $n \geq 1$ and $\xi \in \fM^n$ define
    \begin{gather*}
    \mu(\xi) = \frac{1}{n}\sum \xi_i,
    \qquad
    \mu^*(\xi) = \bigl( \mu(\xi) \bigr)^*,
    \qquad
    S_n = \bigl\{ \xi \in \fM^n : \mu(\xi) = \id \bigr\}.
  \end{gather*}
  \end{enumerate}
\end{dfn}

\begin{thm}
  \label{thm:H01TypeSpaces}
  Let $n \geq 1$.
  \begin{enumerate}
  \item For each $\xi \in S_n$ and $i < n$, the $i$th component $\xi_i$ is Lipschitz with constant at most $n$.
  \item The space $S_n$ is compact.
  \item The map $S_n \rightarrow \fM^n \sslash G$ sending $\xi \rightarrow [\xi]$ is a homeomorphic bijection, with inverse $[\xi] \mapsto \tilde{\xi} = \xi \circ \mu^*(\xi)$.
    Moreover, we have $\tilde{\xi} \circ \mu(\xi) = \xi$.
  \item The space $S_2$ is homeomorphic to the Roelcke completion of $G$.
    Consequently, $G$ is Roelcke-precompact.
  \end{enumerate}
\end{thm}
\begin{proof}
  Assume $\xi \in S_n$, and let $0 \leq s \leq t \leq 1$.
  Then
  \begin{gather*}
    \xi_i(s) \leq \xi_i(t) = nt - \sum_{j\neq i} \xi_j(t) \leq nt - \sum_{j\neq i} \xi_j(s) = n(t-s) + \xi_i(s).
  \end{gather*}
  By Arzelà–Ascoli, $S_n$ is compact.

  For the third item, let $\xi \in \fM^n$ be arbitrary.
  If $t,s \in [0,1]$ are such that $\mu(\xi)(t) = \mu(\xi)(s)$ then necessarily $\xi_i(t) = \xi_i(s)$ for all $i$, so $\tilde{\xi} = \xi \circ \mu^*(\xi)$ belongs to $\fM^n$.
  A direct calculation yields $\tilde{\xi} \circ \mu(\xi) = \xi$, so $[\xi] = [\tilde{\xi}]$.
  Similarly, and $\mu(\tilde{\xi}) = \mu(\xi) \circ \mu(\xi)^* = \id$, so $\tilde{\xi} \in S_n$.
  The map $\xi \mapsto \tilde{\xi}$ is $G$-invariant and continuous, so it factors via $\fM^n \sslash G$.
  Thus the map $S_n \rightarrow \fM^n \sslash G$ is a bijection.
  It is clearly continuous (even contractive), and since $S_n$ is compact, it is a homeomorphism.

  In particular, we have $S_2 \cong \fM^2 \sslash G = \widehat{G}_L^2 \sslash G$, and $\widehat{G}_L^2 \sslash G$ coincides the Roelcke completion for any topological group $G$.
  Thus the Roelcke completion is compact.
\end{proof}

Notice that \autoref{thm:H01TypeSpaces} gives an explicit construction of the Roelcke compactification $R(G)$ as $S_2$ together with the map $G \rightarrow S_2$, $g \rightarrow \widetilde{(\id,g)}$.
Replacing a pair $(\xi,\zeta) \in S_2$ with $\xi - \id = \id - \zeta$, we may further identify $R(G)$ with the space of functions $\xi \colon [0,1] \rightarrow \bR$ which are Lipschitz with constant $1$ and vanish at the endpoints.

From a model-theoretic standpoint, the identity $\tilde{\xi} \circ \mu(\xi) = \xi$ means that a tuple $\xi$ is determined by its type (namely, $\tilde{\xi}$) together with an object that is invariant under permutations of $\xi$ (namely, $\mu(\xi)$).
In other words, the only elementary permutation of a finite set is the identity.

\begin{rmk}
  In classical logic, the property that any elementary permutation of a finite set is the identity is usually a consequence of the existence of a definable linear ordering (but not always: the former does not imply the latter).
  In $\fM$ something similar happens.
  Indeed, for $\xi,\zeta \in \fM$ define $o(\xi,\zeta) = \sup \, (\xi-\zeta)$.
  This is a continuous $G$-invariant predicate, so by \autoref{fct:RyllNardzewski} it is definable.
  Semantically, it is a ``continuous order predicate'': it vanishes if and only if $\zeta \geq \xi$, and otherwise it measures the extent to which $\zeta \ngeq \xi$.
  Accordingly, it satisfies the continuous logic analogues of the axioms of an order relation (in the same way that the axioms of a pseudo-distance are the continuous logic analogues of the axioms of an equivalence relation):
  \begin{itemize}
  \item Reflexivity: $o(x,x) = 0$.
  \item Anti-symmetry: $o(x,y) \vee o(y,x) = 0$ implies $x = y$.
  \item Transitivity: $o(x,z) \leq o(x,y) + o(y,z)$.
  \end{itemize}
  It also satisfies some analogue of linearity:
  \begin{itemize}
  \item ``Linearity'': $o(x,y) + o(y,x) \leq 1$.
  \end{itemize}
  In particular, a classical $\{0=T,1=F\}$-valued relation satisfies the first three axioms if and only if it is an order relation, and all four if and only if it is linear.
\end{rmk}

\begin{qst}
  Assume $\fM$ is a general metric structure with a definable relation satisfying all four axioms, say on some complete type $p$, and say moreover that its supremum there is $1$.
  Is this enough to conclude that every elementary permutation of a finite set of realisations of this type is the identity?
\end{qst}

At any rate, this property of $\fM$ suggests that if $\fM$ is to have WEMI it should also have the stronger EMI, which is what we now aim to prove.

\begin{dfn}
  \label{dfn:RhoGap}
  Let $\rho$ be a definable pseudo-distance on $\fM$.
  Say that an open interval $\emptyset \neq (\alpha,\beta) \subseteq (0,1)$ is a \emph{$\rho$-gap} if there exist $\xi,\zeta \in \fM$ with $\xi \sim_\rho \zeta$ such that for all $s$ either $\xi(s) \leq \alpha$ or $\zeta(s) \geq \beta$ (or both).
  Let $U_\rho \subseteq (0,1)$ be the union of all $\rho$-gaps.
\end{dfn}

Equivalently, $\xi \sim_\rho \zeta$ witness that $(\alpha,\beta)$ is a $\rho$-gap if and only if there exists $s$ such that $\xi(s) \leq \alpha < \beta \leq \zeta(s)$ (take $s = \sup \bigl\{ t : \xi(t) \leq \alpha \bigr\}$).
Therefore, $\alpha \in U_\rho$ if and only if there exists a pair $\xi \sim_\rho \zeta$ such that $\xi(s) < \alpha < \zeta(s)$ for some $s$.

\begin{lem}
  \label{lem:RhoGap}
  Let $\rho$ be a definable pseudo-distance on $\fM$.
  \begin{enumerate}
  \item \label{item:RhoGapUnion}
    The union of any two intersecting $\rho$-gaps is a $\rho$-gap.
  \item \label{item:RhoGapLargest}
    The set $U_\rho$ is the disjoint union of the maximal $\rho$-gaps.
  \item \label{item:RhoGapExtremeOne}
    Let $I = (\alpha,\beta)$ be a maximal $\rho$-gap.
    Define:
    \begin{gather*}
      \xi_I(s) =
      \begin{cases}
        s & s \leq \alpha \text{ or } s \geq \beta, \\
        \alpha & \alpha \leq s \leq \half[\alpha+\beta], \\
        2s - \beta & \half[\alpha+\beta] \leq s \leq \beta.
      \end{cases}
      \qquad
      \zeta_I(s) =
      \begin{cases}
        s & s \leq \alpha \text{ or } s \geq \beta, \\
        2s - \alpha & \alpha \leq s \leq \half[\alpha+\beta], \\
        \beta & \half[\alpha+\beta] \leq s \leq \beta.
      \end{cases}
    \end{gather*}
    Then $\xi_I \sim_\rho \zeta_I$.
    In particular, $\xi_I,\zeta_I$ witness that $I$ is a $\rho$-gap.
  \item \label{item:RhoGapSeparated}
    The complement $K_\rho = (0,1) \setminus U_\rho$ has no isolated points.
  \item \label{item:RhoGapExtreme}
    Let $\xi_\rho$ ($\zeta_\rho$) be the identity on $K_\rho$ and agree with $\xi_I$ ($\zeta_I$) on each maximal $\rho$-gap $I$.
    Then $\xi_\rho \sim_\rho \zeta_\rho$.
    In particular, $\xi_\rho,\zeta_\rho$ witness all $\rho$-gaps.
  \item \label{item:RhoGapCharacterisation}
    We have $\xi \sim_\rho \zeta$ if and only if $\xi(s) \neq \zeta(s)$ implies $\half[\xi(s) + \zeta(s)] \in U_\rho$.
    When $\half[\xi + \zeta] = \id$, i.e., when $(\xi,\zeta) \in S_2$, we have $\xi \sim_\rho \zeta$ if and only if $\xi(s) = \zeta(s) = s$ on $K_\rho$.
  \item \label{item:RhoGapTrivial}
    If $K_\rho = \emptyset$ then $\rho = 0$ is the trivial pseudo-distance and $\fM_\rho$ is a singleton.
  \item \label{item:RhoGapNonTrivial}
    Otherwise, $K_\rho \neq \emptyset$ and there exists $\chi \in \fM$ which is constant on all $\rho$-gaps and strictly increasing on $K_\rho$.
    The pseudo-distance $\rho_\chi(\xi,\zeta) = d(\chi \circ \xi, \chi \circ \zeta)$ is then definable and uniformly equivalent to $\rho$.
    In this case $\fM_\rho$ is isomorphic to $\fM$.
  \end{enumerate}
\end{lem}
\begin{proof}
  For \autoref{item:RhoGapUnion} assume that $(\alpha,\beta)$ and $(\gamma,\delta)$ are intersecting $\rho$-gaps.
  If one contains the other then the union is already a $\rho$-gap, so we may assume that $\alpha < \gamma < \beta < \delta$.
  Let $\xi,\zeta$ (respectively, $\zeta',\chi$) witness that $(\alpha,\beta)$ (respectively $(\gamma,\delta)$) are $\rho$-gaps.
  Since this is merely a property of the types $[\xi,\zeta]$ and $[\zeta',\chi]$, we may assume that $\zeta = \zeta'$ (by \autoref{lem:RoelckePrecompactAmalgamation}).
  Now, $\xi \sim_\rho \chi$ and for all $s$, if $\xi(s) > \alpha$ then $\zeta(s) \geq \beta > \gamma$ ad therefore $\chi(s) \geq \delta$, witnessing that $(\alpha,\delta)$ is a $\rho$-gap.

  For \autoref{item:RhoGapLargest} let $\gamma \in U_\rho$ and let $(\alpha,\beta)$ be the union of all $\rho$-gaps containing it.
  By \autoref{item:RhoGapUnion} and a compactness argument, if $\alpha_n \searrow \alpha$ and $\beta_n \nearrow \beta$ (and $\alpha_0 < \beta_0$) then $(\alpha_n,\beta_n)$ is a $\rho$-gap, say witnessed by $[\xi_n,\zeta_n]$.
  Possibly passing to a sub-sequence we may assume that $[\xi_n,\zeta_n] \rightarrow [\xi,\zeta]$, and then we may further assume that in fact $\xi_n \rightarrow \xi$ and $\zeta_n \rightarrow \zeta$.
  Assume now that $\xi(s) > \alpha$.
  Then for all $n$ large enough we have $\xi_n(s) > \alpha_n$, so $\zeta_n(s) \geq \beta_n$ and therefore $\zeta(s) \geq \beta$ as well, witnessing that $(\alpha,\beta)$ is a $\rho$-gap.
  It is by construction maximal.
  By \autoref{item:RhoGapUnion}, any two distinct maximal $\rho$-gaps are disjoint.

  For \autoref{item:RhoGapExtremeOne} let $\xi,\zeta$ witness that $I$ is a $\rho$-gap.
  Up to a reparameterisation we may assume that
  \begin{gather*}
    \alpha = \inf \, \{t : \xi(t) > \alpha \} < \sup \, \{t : \xi(t) < \beta \} = \beta.
  \end{gather*}
  Then $\zeta(\alpha) \geq \beta$ (since $(\xi,\zeta)$ witness that $(\alpha,\beta)$ is a $\rho$-gap) and $\zeta(\beta) \leq \beta$ (otherwise $(\xi,\zeta)$ would witness that $\beta \in U_\rho$), so $\zeta(s) = \beta$ for all $s \in [\alpha,\beta]$.
  Therefore $\zeta = \zeta \circ \xi_I = \zeta \circ \zeta_I$.
  If we let $\xi' = \xi \circ \xi_I$, $\xi'' = \xi \circ \zeta_I$, we have $[\xi',\zeta] = [\xi'',\zeta] = [\xi,\zeta]$, so $\xi' \sim_\rho \zeta \sim_\rho \xi''$.
  We also have $\xi'(s) = \xi''(s)$ for all $s \notin (\alpha,\beta)$, and for $s = \half[\alpha+\beta]$ we have $\xi'(s) = \xi(\alpha) = \alpha$ and $\xi''(s) = \xi(\beta) = \beta$.
  Therefore $[\xi',\xi''] = [\xi_I,\zeta_I]$ and thus $\xi_I \sim_\rho \zeta_I$.

  For \autoref{item:RhoGapSeparated}, assume that $\beta$ is an isolated point of $K_\rho$, i.e, that there exist maximal $\rho$-gaps $I = (\alpha,\beta)$ and $J = (\beta,\gamma)$.
  Notice that $\zeta_I$ is strictly increasing except on $[\half[\alpha+\beta],\beta]$ where it is equal to $\beta$, and similarly $\xi_J$ with interval $[\beta,\half[\beta+\gamma]]$.
  Therefore there exists $g \in G$ such that $\zeta_I = \xi_J \circ g$, so $\xi_I \sim_\rho \zeta_I \sim_\rho \zeta_J \circ g$.
  Let $\half[\alpha+\beta] < s < \beta$, so $\beta < g(s) < \half[\beta+\gamma]$ and $\xi_I(s) < \beta < \zeta_J \circ g(s)$, witnessing that $\beta \in U_\rho$ after all.

  For \autoref{item:RhoGapExtreme} enumerate the maximal $\rho$-gaps as $\{I_n\}$ (there are at most countably many -- if there are only finitely many add a tail of empty sets to the sequence).
  For each $n$ let $\xi_n$ agree with $\xi_\rho$ on $\bigcup_{m \geq n} I_m$, with $\zeta_\rho$ on $\bigcup_{m < n} I_m$, and with both (i.e., with the identity) elsewhere.
  Then $[\xi_n,\xi_{n+1}] = [\xi_{I_n},\zeta_{I_n}]$, so by induction $\xi_\rho \sim_\rho \xi_n$.
  Since $\xi_n \rightarrow \zeta_\rho$ we conclude that $\xi_\rho \sim_\rho \zeta_\rho$.

  For \autoref{item:RhoGapCharacterisation}, observe first that if $\xi \sim_\rho \zeta$ and $\xi(s) \neq \zeta(s)$, say $\xi(s) < \zeta(s)$, then $\bigl( \xi(s),\zeta(s) \bigr)$ is a $\rho$-gap containing $\half[\xi(s) + \zeta(s)]$.
  Conversely, assume the condition holds, and without loss of generality we may assume that $(\xi,\zeta) \in S_2$, i.e., that $\half[\xi+\zeta] = \id$.
  Our hypothesis implies that $\xi$ and $\zeta$ agree with the identity outside $U_\rho$.
  It follows that $[\xi \circ \xi_\rho,\zeta_\rho] = [\zeta \circ \xi_\rho,\zeta_\rho] = [\xi_\rho,\zeta_\rho]$, so $\xi \circ \xi_\rho \sim_\rho \zeta_\rho \sim_\rho \zeta \circ \xi_\rho$ and therefore $\xi \sim_\rho \zeta$.

  Item \autoref{item:RhoGapTrivial} is clear.
  For \autoref{item:RhoGapNonTrivial}, the construction is fairly standard.
  One fixes an enumeration $\bigl\{ I_n = (\alpha_n,\beta_n) \bigr\}$ of maximal $\rho$-gaps and a sequence $r_n \in [0,1]$ such that $r_n < r_m \Longleftrightarrow \beta_n < \alpha_m$ and $r_n = 0$ if $\alpha_n = 0$, $r_n = 1$ if $\beta_n = 1$.
  Then one defines $\chi_m$ to be equal to $r_n$ on $I_n$ for $n < m$ and to increase with a constant slope between adjacent gaps among $\{I_n\}_{n<m}$.
  The sequence $\chi_m$ converges uniformly to $\chi$ which is as desired.
  It follows from \autoref{fct:RyllNardzewski} that $\rho_\chi$, being $G$-invariant and continuous, is definable.
  By \autoref{item:RhoGapCharacterisation}, if $(\xi,\zeta) \in S_2$ then $\xi \sim_{\rho_\chi} \zeta \Longleftrightarrow \xi \sim_\rho \zeta$, so the same holds to arbitrary $\xi,\zeta \in \fM$.
  By a compactness argument, the two pseudo-metrics must be uniformly equivalent.
  In other words, sending the class of $\xi$ modulo $\rho$ to $\chi \circ \xi$ gives rise to an isomorphism of complete $G$-spaces $\fM_\rho \cong \fM$.
\end{proof}

Let us observe that the reasoning of \autoref{thm:H01TypeSpaces} also applies to infinite tuples in the following sense.
For $\xi \in \fM^\bN$ let $\eta(\xi) = \sum_i 2^{-i-1} \xi_i \in \fM$ and $\hat{\xi}_i = \xi_i \circ \eta^*(\xi)$.
Then $\hat{\xi}_i \in \fM$ has Lipschitz constant at most $2^{i+1}$, so $\hat{\xi}$ belongs to some fixed compact set (the space of $\bN$-types), and $\eta(\hat{\xi}) = \id$.
Moreover, we may recover the original tuple as $\xi = \hat{\xi} \circ \eta(\xi)$ (the same works with any family of strictly positive coefficients adding up to one).

\begin{thm}
  \label{thm:H01NoTrans}
  The structure $\fM$ has EMI -- in fact, any countable set of imaginary elements is interdefinable with either $\emptyset$ or some (singleton) $\xi \in \fM$.
  It therefore satisfies the hypotheses of \autoref{thm:NoTrans}, so every transitive complete $\Homeo^+[0,1]$-space is a singleton.
\end{thm}
\begin{proof}
  Let $G = \Homeo^+[0,1]$ and let $X$ be a minimal complete $G$-space.
  We need to show that if $X$ is not a singleton, then there exists $\xi \in \fM$ such that $[\xi]$ is isomorphic, as a $G$-space, to $X$.

  We know that $X$ is isomorphic to $[a]$ for some imaginary $a \in \fM_\rho$, where $\rho$ is a definable pseudo-distance on $\fM^\bN$.
  While $a$ itself need not necessarily be a $\sim_\rho$-class (after all, $\fM_\rho$ is a completion), there does exist $a' \in [a]$ which is.
  Therefore, possibly replacing $a$ with $a'$, we may assume that $a$ is the $\sim_\rho$-class of some infinite tuple $\zeta \in \fM^\bN$.
  Applying the discussion above, we have $\zeta = \hat{\zeta} \circ \eta(\zeta) = \bigl( \hat{\zeta}_i \circ \eta(\zeta) \bigr)_{i\in\bN}$.
  Define a pseudo-distance $\rho'$ on $\fM$ by:
  \begin{gather*}
    \rho'(\xi,\xi') = \rho(\hat{\zeta} \circ \xi, \hat{\zeta} \circ \xi').
  \end{gather*}
  By \autoref{fct:RyllNardzewski} this is definable.
  The map sending the $\sim_{\rho'}$-class of $\xi$ in $\fM_{\rho'}$ to the $\sim_\rho$-class of $\hat{\zeta} \circ \xi$ in $\fM_\rho$ is isometric and $G$-invariant, extending to an isometric map of $G$-spaces $\fM_{\rho'} \rightarrow \fM_\rho$, which sends, in particular, the class of $\eta(\zeta)$ to $a$.
  Since $\fM$ is a minimal $G$-space, so is $\fM_{\rho'}$, and the entire image must lie in a single orbit closure, namely in $[a]$.
  We thus obtain an isomorphism of the complete $G$-space $\fM_{\rho'}$ with $[a]$ and therefore with $X$.
  Since $\rho'$ is not trivial ($X$ is not a singleton), $\fM_{\rho'}$ is further isomorphic, by \autoref{lem:RhoGap}, to $\fM$.
  By \autoref{lem:Aleph0CategoricalEMI}, $\fM$ has EMI.
\end{proof}

\section{A question regarding the connection with discrete uniform distance}
\label{sec:DiscreteUniformDistance}

Any topological group $G$ carries a unique compatible left-invariant uniformity, which gives rise to a coarsest bi-invariant uniformity refining the topology (although not compatible, in general).
When $G$ is metrisable (e.g., Polish), it carries a compatible left-invariant metric $d_L$, and the bi-invariant uniformity is then given by $d_u(g,h) = \sup_{f \in G} \, d_L(gf,hf)$.
This is the distance of uniform convergence of the left action $G \curvearrowright (G,d_L)$ (or $G \curvearrowright \widehat{G}_L$).
If $G = \Aut(\fM)$ for an $\aleph_0$-categorical structure $\fM$, the bi-invariant uniformity further agrees with the uniform convergence of the action on $\fM$ (or, if $\fM$ is many-sorted, the product uniformity of the actions on the separate sorts), whence the terminology ``uniform distance'' in \cite{BenYaacov-Berenstein-Melleray:TopometricGroups} (notice that Pestov \cite{Pestov:Dynamics} mentions a coarser distance on $\Aut^*(\mu)$ which he calls ``uniform'' and which is not bi-invariant).

When $G$ is the automorphism group of a classical $\aleph_0$-categorical structure (in one, or finitely many, sorts) this uniform distance $d_u$ is always discrete, but for a metric structure we should expect it to be non-discrete.
This plays a crucial role in \cite{BenYaacov-Berenstein-Melleray:TopometricGroups}, where we show that certain Polish groups (the unitary group $U(\ell_2)$, the isometry group of the Urysohn sphere $\Iso(\bU_1)$, the group $\Aut(\mu)$ of measure-preserving transformations of $[0,1]$) have ample generics up to the uniform distance, even though they do not have precise ample generics (i.e., up to the discrete distance).
The uniform distance is relatively easy to compute explicitly, and of the Roelcke-precompact Polish groups considered so far, there are exactly three which, despite our expectations, have discrete uniform distance: $\Aut^*(\mu)$, $\Homeo^+[0,1]$ and $\Aut(\bG)$.
For the first two, this is fairly immediate:

\begin{prp}
  The uniform distance on the groups $\Aut^*(\mu)$ and $\Homeo^+[0,1]$ is discrete.
\end{prp}
\begin{proof}
  If $g \in \Aut^*(\mu)$ is not the identity there exists a measurable set $A$ such that $\mu(A) > 0$ and $gA \cap A = \emptyset$.
  Let $\xi = \bone_A/\mu(A)$ be the normalised characteristic function.
  Then $\xi$ belongs to the unit ball of $L_1(\mu)$ and $d(\xi,g\xi) > 1$.

  If $g \in \Homeo^+[0,1]$ is not the identity, we may assume that $g(t) > t$ for some $t$ (otherwise replace $g$ with $g^{-1}$).
  Let $\xi \colon [0,1] \rightarrow [0,1]$ be continuous, increasing, equal to $0$ on $[0,t]$ and to $1$ on $[g(t),1]$.
  Then $d(g\xi,\xi) = d(\xi \circ g^{-1},\xi) = 1$.

  In either case we have $G = \Aut(\fM)$, where $\fM$ is $\aleph_0$-categorical, so $d_u$ is uniformly equivalent with uniform convergence on $\fM$, which by the above is discrete.
\end{proof}

The case of $\Aut(\bG)$ requires a more substantial argument.
In fact, we give two arguments.
The first argument was communicated to the author by Melleray and Tsankov, whom we thank for the permission to include it here.

\begin{prp}
  \label{prp:AutGurarijDiscreteUniformDistance}
  The uniform distance on $\Aut(\bG)$ is discrete.
\end{prp}

\begin{fct}[{\cite{Lazar-Lindenstrauss:DualL1}}]
  For a Banach space $X$, the following are equivalent:
  \begin{enumerate}
  \item The dual $X^*$ is isometric to $L_1(\mu)$ for some measure $\mu$.
  \item The space $X$ is an $a^\infty$ space: that is to say that for any finite set $A \subseteq X$ and $\varepsilon > 0$ there exists $n$ and an operator $T\colon \ell_\infty(n) \rightarrow X$ such that
    \begin{itemize}
    \item For all $y \in \ell_\infty(n)$ we have $(1+\varepsilon)^{-1} \|y\| \leq \| T y \| \leq (1+\varepsilon) \| y \|$.
    \item For all $x \in A$ we have $d\bigl( x, T\ell_\infty(n) \bigr) < \varepsilon$.
    \end{itemize}
  \end{enumerate}
\end{fct}

\begin{proof}[First proof of \autoref{prp:AutGurarijDiscreteUniformDistance}]
  Let $g \in \Aut(\bG)$, and we need to show that if $g \neq \id$ then $\|g-\id\|$ is bounded away from $0$.
  The same $g$ also acts isometrically on $\bG^*$ as $g^* \lambda = \lambda \circ g$, and one easily checks that
  \begin{align*}
    \|g-\id_\bG \|
    &
      = \sup \, \bigl\{ \|\lambda g x - \lambda x\| : x \in \bG, \, \lambda \in \bG^*, \, \|x\| = \|\lambda\| = 1 \bigr\}
    \\ &
         = \|g^* - \id_{\bG^*}\|.
  \end{align*}
  It is fairly immediate to check that the Gurarij space is $a^\infty$, so $\bG^*$ is isometric to some $L_1(\mu)$.

  Since $L_1(\mu)$ is a dual, by Krein-Milman its unit ball is the weak$^*$-closed convex hull of its extreme points.
  An extreme point of the unit ball of $L_1(\mu)$ is necessarily of the form $\pm \bone_A / \mu(A)$ for some atom $A$ of $\mu$.
  The distance between two distinct extreme points is $2$, so $\|g - \id_\bG\| = \|g^* - \id_{\bG^*}\| = 2$.
\end{proof}

This proof leaves the author somewhat unsatisfied, since it uses quite a few ``black boxes'' and does not actually produce $x \in \bG$ such that $\|gx - x\| \approx 2$.
We therefore propose a more explicit, even if somewhat longer, argument:

\begin{lem}
  \label{lem:AutomorphismMovesNormingFunctional}
  Let $E$ be a separable Banach space, $g \in \Aut(E)$ an isometric automorphism which is not the identity.
  Then there exists $x \in E$ and $\lambda \in E^*$ such that $\lambda$ norms $x$ but not $gx$.
\end{lem}
\begin{proof}
  Let $S \subseteq E$ denote the set of smooth unit vectors.
  By Mazur \cite{Mazur:Konvexe}, this is a dense subset of the unit ball.
  Let $S^* \subseteq E^*$ denote the set of unit linear functionals which norm elements of $S$.
  We claim that $\cco(S^*)$ is the unit ball of $E^*$ (closure being in the weak$^*$ topology).
  Indeed, if not, there there exists a unit linear functional $\lambda \notin \cco(S^*)$, and by Hahn-Banach (plus the fact that a weak$^*$-continuous linear functional on $E^*$ arises from a member of $E$) there exists a unit vector $u \in E$ such that $\lambda u > \sup\, \bigl\{ \mu u : \mu \in \cco(S^*) \bigr\}$.
  The same is true in some norm neighbourhood of $u$, and since $S$ is dense, we may assume that $u \in S$.
  But $\sup\, \bigl\{ \mu u : \mu \in S^* \bigr\} = 1$, a contradiction.

  If follows that $g$ cannot be the identity on $S^*$: there exists $\lambda \in S^*$, norming some $x \in S$, such that $\lambda g \neq \lambda$.
  Since $x$ is smooth, $\lambda g x  < 1$, i.e., $\lambda$ does not norm $gx$.
\end{proof}

\begin{proof}[Second proof of \autoref{prp:AutGurarijDiscreteUniformDistance}]
  This is based on the characterisation of $1$-types as \emph{convex Katětov functions} as in \cite{BenYaacov:UniversalGurarijIsometryGroup}.

  Let $g \in \Aut(\bG)$ be other than the identity.
  Let $x \in \bG$ and $\lambda \in \bG^*$ be of norm one, such that $\lambda x = 1$ and $\lambda g x < 1$, as per \autoref{lem:AutomorphismMovesNormingFunctional}.

  Let $E \subseteq \bG$ be the subspace generated by $x$ and $y = gx$.
  Let $\varepsilon = 1 - \lambda y > 0$.
  Thus $\|x - ty\| \geq 1 - t(1-\varepsilon) = 1 - t + t\varepsilon$.
  Let $M = 2/\varepsilon$, and define $f \colon E \rightarrow \bR$ by:
  \begin{gather*}
    f(z) = \inf_{0 \leq t \leq 1} 1 + (M - 2) t + \|z - t M y\|.
  \end{gather*}
  The function $f$ is convex and Lipschitz with constant $1$.
  In addition, for $z,z' \in E$ and $t,t' \in [0,1]$ we have
  \begin{align*}
    & 1 + (M - 2) t + \|z - t M y\| + 1 + (M - 2) t' + \|z' - t' M y\| \\
    & \qquad \geq \bigl\| (z-z') - (t-t') M y \bigr\| + 2 + (M-2)(t-t') \\
    & \qquad \geq \bigl\| (z-z') - (t-t') M y \bigr\| + M(t-t') \\
    & \qquad \geq \|z-z'\|,
  \end{align*}
  whereby $f(z)+f(z') \geq \|z-z'\|$.
  We conclude that $f$ is a convex Katětov function on $E$, so by \cite{BenYaacov:UniversalGurarijIsometryGroup} there exists a $1$-point extension $F \supseteq E$, generated over $E$ by, say, $w$, such that $\|z-w\| = f(z)$ for all $z \in E$.

  By the Gurarij property, for any $\delta > 0$ and finite set $A \subseteq E$ there exists $w' \in \bG$ such that $\Bigl| \|w'-z\| - f(z) \Bigr| < \delta$ for all $z \in A$.
  Let us do this for the finite set $\{0, Mx, My\}$.
  On the one hand, $f(My) \leq 1 + (M-1) \cdot 1 + \| My - My \| = M - 1$ (in fact there is equality).
  On the other hand, for $0 \leq t \leq 1$:
  \begin{align*}
    1 + (M-2) t + \| Mx- tMy\|
    &
      \geq 1 + (M-2) t + M(1 - t + t \varepsilon)
    \\ &
         = 1 - 2t + M + M t \varepsilon
    \\ &
         = M + 1.
  \end{align*}
  Thus $f(Mx) \geq M + 1$ and
  \begin{align*}
    \|gw' - w'\|
    &
      \geq \|gw' - My\| - \|w' - My\|
    \\ &
         = \|w' - Mx\| - \|w' - My\|
    \\ &
         \geq M + 1 - (M-1) - 2\delta
    \\ &
         \geq 2 - 2\delta.
  \end{align*}
  Since $\|w'\| \in [1-\delta,1+\delta]$, this suffices to prove that $\|g - \id\| = 2$.
\end{proof}

Thus, among the automorphism groups of $\aleph_0$-categorical structures studied so far, those which present the ``no transitive action'' pathology, or suspected of presenting it, are exactly those which are automorphism groups of ``essentially non-discrete'' $\aleph_0$-categorical structures and yet have discrete uniform distance.

\begin{qst}
  Let $G$ be a connected Roelcke-precompact Polish group.
  \begin{enumerate}
  \item Does the existence of a non-trivial transitive complete $G$-space imply that the uniform distance on $G$ is not discrete?
  \item Does the converse implication hold?
  \end{enumerate}
\end{qst}

\begin{rmk}
  In the context of closed permutation groups, i.e., closed subgroups of $\fS_\infty$, the discrete uniform distance allows us to distinguish oligomorphic groups from among the pro-oligomorphic ones.

  Indeed, recall from Tsankov \cite{Tsankov:OligomorphicGroupRepresentation} that,
  \begin{itemize}
  \item A topological group $G$ is \emph{oligomorphic} if (by definition) it can be realised as a closed permutation group of some countable set $X$ with oligomorphic action, i.e., such that $X^n/G$ is finite for all $n$
  \item It is pro-oligomorphic (i.e., an inverse limit of oligomorphic groups) if and only if it is a closed subgroup of $\fS_\infty$ and Roelcke-precompact.
  \end{itemize}

  First, in terms of automorphism groups, one checks that $G$ is oligomorphic (pro-oligomorphic) if and only if it is the automorphism group of an $\aleph_0$-categorical structure in a countable language and finitely (countably) many sorts $\fM$.

  Second, a pro-oligomorphic group is oligomorphic if and only if its uniform distance is discrete.
\end{rmk}

\section{Transitive complete $G$-spaces}
\label{sec:TransitiveAction}

This section is somewhat complementary to earlier sections, in that we characterise complete $G$-spaces which are transitive (for a Roelcke-precompact $G$).

\begin{lem}[{\cite[Lemma~3.8]{BenYaacov-Tsankov:WAP}}]
  \label{lem:ApproximatelyInvertible}
  Let $G$ be Roelcke-precompact and $R = \widehat{G}_L^2 \sslash G$ its Roelcke completion.
  For $\varepsilon > 0$ let $V_\varepsilon \subseteq R$ consist of all $p \in R$ such that $d\bigl( G p, 1 \bigr) < \varepsilon$, where $g[\xi,\zeta] = [\xi g^{-1},\zeta]$.
  Then $\bigcap_{\varepsilon>0} V_\varepsilon = \widehat G_L$ and $\bigcap_{\varepsilon>0} (V_\varepsilon \cap V_\varepsilon^*) = G$, where $[\xi,\zeta]^* = [\zeta,\xi]$.
\end{lem}

\begin{thm}
  \label{thm:TransitiveAction}
  Let $G$ be Roelcke-precompact, $G \curvearrowright X$ be a minimal continuous action by isometry on a complete space $X$.
  Then the following are equivalent:
  \begin{enumerate}
  \item The action $G \curvearrowright X$ is transitive.
  \item
    \label{item:TransitiveActionRyllNardzewski}
    The space $\widehat G_L \sslash G_x$ is compact for some (all) $x \in X$.
  \item
    \label{item:TransitiveActionContinuity}
    For all $\varepsilon > 0$ and $x \in X$ there exists $\delta > 0$ such that whenever $y \in X$ and $d(x,y) < \delta$ there exist $\xi,\zeta \in \widehat{G}_L$ with $\xi x = \zeta y$ and $d_L(\xi,\zeta) < \varepsilon$.
  \end{enumerate}
\end{thm}
\begin{proof}
  \begin{cycprf}
  \item
    By Effros \cite{Effros:TransformationGroups}, we may assume that $X = G/H$ and $x = H \in G/H$.
    Let $[\xi_n]_H$ be a sequence in $\widehat G_L \sslash H$.
    Since $G$ is Roelcke-precompact we may assume that $[1,\xi_n] \rightarrow [\zeta,\xi]$ in $R$.
    Since $\widehat{G}_L$ acts on $G/H$, we have $\zeta H = f H$ for some $f \in G$.

    There exist $g_n \in G$ such that $g_n \xi_n \rightarrow \xi$ and $g_n \rightarrow \zeta$.
    Therefore $g_n H \rightarrow f H$, i.e. $g_n h_n^{-1} \rightarrow f$ for some $h_n \in H$.
    Then $h_n \xi_n \rightarrow f^{-1} \xi$, i.e., $[\xi_n]_H \rightarrow [f^{-1} \xi]_H$, whence compactness.
  \item
    Assume $\widehat G_L \sslash H$ is compact where $H = G_x$ for some $x \in X$.
    Consider the space $\widetilde{X} = \widehat{G}_L \times X \sslash G$.
    It is compact, and the map $\pi\colon \widehat{G}_L \rightarrow \widetilde{X}$, $\pi \xi = [\xi,x]$, factors via a continuous map $\overline{\pi}\colon \widehat{G}_L \sslash H \rightarrow \widetilde{X}$.
    Since $\overline{\pi}$ has dense image, it is surjective, and one checks easily that $[\xi,x] = [1,y]$ if and only if $\xi y = x$.
    Assume for a contradiction that there exist $\varepsilon > 0$ and $z,y_n \in X$ such that $y_n \rightarrow z$ and yet $d_L(\xi,\zeta_n) \geq \varepsilon$ whenever $\xi z = \zeta_n y_n$.
    Choose $\zeta_n$ such that $\zeta_n y_n = x$, which is possible since $\overline{\pi}$ is surjective.
    By hypothesis we may assume that $[\zeta_n]_H \rightarrow [\zeta]_H$, i.e., $h_n \zeta_n \rightarrow \zeta$ with $h_n \in H$.
    Then $\zeta z = x$ as well, a contradiction.
  \item[\impfirst]
    Let $x,y \in X$ and let $R_{x,y} = \bigl\{ [\xi,\zeta] \in R : \xi x = \zeta y \bigr\}$, observing that the condition $\xi x = \zeta y$ only depends on the class $[\xi,\zeta]$.
    We want to show that $G \cap R_{x,y} \neq \emptyset$.
    By \autoref{lem:ApproximatelyInvertible}, it will suffice to show that $V_\varepsilon \cap R_{x,y}$ is dense in $R_{x,y}$ for all $\varepsilon > 0$.
    Indeed, let $[\xi,\zeta] \in R_{x,y}$ and $\varepsilon > 0$.
    We may assume that $\zeta$ is arbitrarily close to $1$, and choose $g \in G$ arbitrarily close to $\xi$.
    In particular, $gx$, $\xi x = \zeta y$ and $y$ are arbitrarily close.
    By hypothesis there exists $[\rho,\chi] \in R_{gx,y}$ such that $d_L(\rho,\chi)$ is small, say $d_L(g,\xi) + d_L(1,\zeta) + d_L(\rho,\chi) < \varepsilon$.
    Then $[\rho g,\chi] \in R_{x,y}$ and
    \begin{gather*}
      d\bigl( g [\rho g,\chi], 1 \bigr) = d\bigl( [\rho,\chi], [\rho,\rho] \bigr) < \varepsilon,
    \end{gather*}
    so $[\rho g,\chi] \in V_\varepsilon$.
    In addition,
    \begin{gather*}
      d\bigl( [\rho g,\chi], [\xi,\zeta] \bigr) \leq  d\bigl( [\rho g,\chi], [\rho g, \rho] \bigr) + d\bigl( [g,1], [\xi,\zeta] \bigr) < \varepsilon,
    \end{gather*}
    whence the desired result.
  \end{cycprf}
\end{proof}

The reader may recognise in condition \autoref{item:TransitiveActionRyllNardzewski} the criterion of \autoref{fct:RyllNardzewski} (for the structure in which $x$ is named by a constant), while condition \autoref{item:TransitiveActionContinuity} echoes the notion of $d$-finiteness from \cite{BenYaacov-Usvyatsov:dFiniteness}.

Let us consider a Polish group $G$ and a minimal complete $G$-space $G \curvearrowright X$.
If $X$ contains a $G_\delta$ orbit $Gx$ (necessarily dense) then this orbit is homeomorphic to $G/H$ where $H = G_x$, by Effros \cite{Effros:TransformationGroups}.
Conversely, if $G/H$ admits a $G$-invariant compatible distance $d$, then the completion $X = \widehat{(G/H,d)}$ is a minimal complete $G$-space, containing a $G_\delta$ orbit $G/H$.
However, such a distance need not always exist (e.g., for $G = S_\infty = \Aut(\bQ,=)$ and $H = \Aut(\bQ,<)$).
While a criterion for the existence of such a distance can be recovered from \cite{BenYaacov-Melleray:Isometrisable}, a much more direct approach exists (this criterion was also observed independently by Tsankov).

\begin{lem}
  \label{lem:NormallyClosed}
  Let $G$ be a Polish group and $H \leq G$ a closed subgroup.
  Then the following are equivalent:
  \begin{enumerate}
  \item The quotient space $G/H$ admits a compatible $G$-invariant distance.
  \item For every open neighbourhood of identity $1 \in U \subseteq G$ there exists a smaller neighbourhood $V$ such that $VH \subseteq HU$ (and thus $HVH \subseteq HU$).
  \end{enumerate}
\end{lem}
\begin{proof}
  Assume that $d$ be a compatible $G$-invariant distance on $G/H$.
  Let $U$ be a neighbourhood of $1$, which we may assume to be symmetric.
  Then $UH \subseteq G$ is open, and so is $U \cdot H \subseteq G/H$, so we have $B(H,r) \subseteq U \cdot H$ for some $r > 0$.
  Let $V = \bigl\{ g \in G : d(gH,H) < r \bigr\}$.
  Then $V$ is open, $VH = V \subseteq UH$, and since $d$ is $G$-invariant, $V = V^{-1}$, so $HV = V \subseteq UH$.

  Conversely, assume that the second condition holds.
  For open $1 \in V = V^{-1} \subseteq G$ define $W_V \subseteq (G/H)^2$ as $\bigl\{ (gH,fH) : f^{-1}g \in HVH \bigr\}$.
  The family of all such $W_V$ forms a uniform structure on $G/H$, with $G$-invariant entourages and countable basis, so it admits a $G$-invariant metrisation.
  In addition, any neighbourhood of $H$ in $G/H$ is of the form $U \cdot H$ for $1 \in U \subseteq G$ open, and taking $V$ as in the hypothesis we see that $\bigl\{ gH : (H,gH) \in W_V \bigr\} \subseteq U \cdot H$, so it is compatible with the topology on $G/H$.
\end{proof}

\begin{cor}
  Let $G$ be a Roelcke-precompact Polish group, and $H \leq G$ a closed subgroup such that
  \begin{enumerate}
  \item $H$ satisfies the equivalent conditions of \autoref{lem:NormallyClosed}.
  \item $\widehat{G}_L \sslash H$ is compact.
  \end{enumerate}
  Then there exists a transitive complete $G$-space $X$ such that $H$ is the stabiliser of some $x \in X$.
  Moreover, $X$ is determined by $H$, up to uniformly continuous isomorphism of $G$-spaces, and every transitive $G$-space is of this form.
\end{cor}
\begin{proof}
  If $H$ satisfies the two hypotheses then $G/H$ admits a $G$-invariant distance, and it is complete by \autoref{thm:TransitiveAction}.
  Conversely, if $X$ is a complete transitive $G$-space and $H = G_x$ for some $x \in X$ then $G/H \rightarrow X$, $gH \mapsto gx$, is a homeomorphism by Effros \cite{Effros:TransformationGroups}, and this determines a unique compatible $G$-invariant uniform structure on $X$.
\end{proof}

We conclude by showing that the natural action of a Roelcke-precompact group on its Bohr compactification is transitive (subsequent to the circulation of a draft of the present paper, Todor \textsc{Tsankov} announced a more direct proof).
Recall that the \emph{Bohr compactification} of a topological group $G$ is a universal continuous morphism into a compact group $\varphi\colon G \rightarrow B$: any other continuous morphism into a compact group factors uniquely through $\varphi$.
The Bohr compactification always exists and has dense image.
We may equip it with a compatible invariant distance, rendering it a minimal complete $G$-space (via $g \cdot b = \varphi(g)b$).

\begin{thm}
  \label{thm:SurjectiveBohrCompactification}
  Let $G$ be a Roelcke-precompact Polish group and $\varphi\colon G \rightarrow B$ its Bohr compactification.
  Then $\varphi$ is surjective.
\end{thm}
\begin{proof}
  For $\xi,\zeta \in \widehat{G}_L$ define $\xi_B = \xi \cdot 1_B$ and $\psi(\xi,\zeta) = \xi_B^{-1} \zeta_B$.
  This map $\psi$ is continuous and $G$-invariant, so it factors as $\overline{\psi} \circ \pi$, where $\pi(\xi,\zeta) = [\xi,\zeta] \in R$ and $\overline{\psi}\colon R \rightarrow B$ is continuous.
  Since $\overline{\psi}$ has dense image and $R$ and $B$ are compact, $\overline{\psi}$ is a topological quotient map.
  Since $\pi$ is one as well, so is the composition $\psi$.
  Define
  \begin{gather*}
    N\colon B \rightarrow \bR, \qquad N(b) = \inf \, \bigl\{ d_L(\chi,\rho) : \psi(\chi,\rho) = b \bigr\}.
  \end{gather*}
  Fix $\xi,\zeta \in \widehat{G}_L$, and consider a pair $\chi,\rho$ such that $\psi(\chi,\rho) = \psi(\xi,\zeta)$.
  By \autoref{lem:RoelckePrecompactAmalgamation} we may find $s,t,u,v \in \widehat{G}_L$ such that $s\xi = u\chi$ and $t\zeta = v\rho$, and by \autoref{lem:RoelckePrecompactAmalgamation} again we may assume that $u = v$.
  Then $d_L(\chi,\rho) = d_L(u\chi,u\rho) = d_L(s\xi,t\zeta)$ and $u_B^{-1} s_B = \chi_B \xi_B^{-1} = \rho_B \zeta_B^{-1} = u_B^{-1} t_B$, so $s_B = t_B$ and
  \begin{gather*}
    N \circ \psi(\xi,\zeta) = \inf\, \bigl\{ d_L(s\xi,t\zeta) : s,t \in \widehat{G}_L, \, s_B = t_B \bigr\}.
  \end{gather*}
  We conclude that $N \circ \psi$ is continuous on $\widehat{G}_L^2$.
  Since $\psi$ is a topological quotient, $N$ is continuous.
  Therefore condition \autoref{item:TransitiveActionContinuity} of \autoref{thm:TransitiveAction} holds, so $G \curvearrowright B$ is transitive, i.e., $\varphi$ is surjective.
\end{proof}

\begin{rmk}
  \label{rmk:SurjectiveBohrEquivalent}
  Let $G$ be a Roelcke-precompact Polish group and $\fM$ an $\aleph_0$-categorical structure with $G = \Aut(\fM)$.
  Then the following assertions are easily shown to be equivalent:
  \begin{enumerate}
  \item The Bohr compactification $\varphi\colon G \rightarrow B$ is surjective.
  \item Any compact minimal metric $G$-space is transitive.
  \item \autoref{prp:Aleph0CategoricalNamedACL} holds for $\fM$.
  \end{enumerate}
  Indeed:
  \begin{cycprf}[1]
  \item
    Let $X$ be a compact minimal metric $G$-space.
    Then $\Iso(X)$ is compact and the action $G \curvearrowright X$ induces a continuous map $G \rightarrow \Iso(X)$ which must factor through $B$.
    Since $B$ is compact and the action $B \curvearrowright X$ is minimal, it is in fact transitive.
    Since $G \rightarrow B$ is surjective, $G \curvearrowright X$ is transitive.
  \item
    Let $a \in \fM^\meq$ be such that $\dcl^\meq(a) = \acl^\meq(\emptyset)$ (such $a$ exists by \autoref{lem:Countable}).
    Then $a \in \acl^\meq(\emptyset)$, so $[a]$ is compact and $G \curvearrowright [a]$ is transitive, by hypothesis.
    By \autoref{lem:Aleph0CategoricalNamedTransitive} again, $(\fM,a)$ is $\aleph_0$-categorical, i.e., $\bigl( \fM, \acl^\meq(\emptyset) \bigr)$ is.
  \item[\impfirst]
    By \autoref{fct:Aleph0CategoricalImaginary}, we may identify $B$ with a metric imaginary sort in $\fM$.
    Since $B$ is compact we have $B \subseteq \acl^\meq(\emptyset)$, so by hypothesis (namely, by \autoref{prp:Aleph0CategoricalNamedACL}), $(\fM,1_B)$ is $\aleph_0$-categorical as well.
    By \autoref{lem:Aleph0CategoricalNamedTransitive} $G \curvearrowright B$ is transitive.
  \end{cycprf}

  In other words, we already had a model-theoretic argument for \autoref{thm:SurjectiveBohrCompactification} -- a highly convoluted one, though, given the pure topological nature of the statement.
\end{rmk}

\providecommand{\bysame}{\leavevmode\hbox to3em{\hrulefill}\thinspace}

\end{document}